\documentclass[reqno]{amsart}
\usepackage[english]{babel}
\usepackage{amscd,amssymb,amsmath,amsfonts,latexsym,amsthm}
\usepackage{inputenc}
\usepackage{graphicx}
\newtheorem{thm}{Theorem}

\newtheorem{prop}{Proposition}

\newtheorem{rem}{Remark}
\begin{document}

\title[Outer derivations of some solvable Lie algebras]{Outer derivations of complex solvable Lie algebras with nilradical of maximal rank}
\author{K.K. Abdurasulov, R.Q. Gaybullayev. }
\address{[K.\ K.\ Abdurasulov] V.I.Romanovskiy Institute of Mathematics Uzbekistan Academy of Sciences, Tashkent 100170, Uzbekistan.} \email{abdurasulov0505@mail.ru}
\address{[R.\ Q.\ Gaybullayev] National University of Uzbekistan, Tashkent, Uzbekistan.} \email{r$_{-}$gaybullaev@mail.ru}

\begin{abstract} In the present paper we prove the existence outer derivations finite-dimensional solvable Lie algebras with nilradical of maximal rank and complementary subspace to nilardical of dimension less than rank of the nilradical.
\end{abstract}



\maketitle


Derivations play an important role in the theory of Lie algebras. They are two kinds of derivations (inner and outer derivations). It is know that every nilpotent Lie algebra over a field of arbitrary characteristic has an outer derivation, while any derivation of semisimple Lie algebras is inner \cite{jaci}. In the case of solvable Lie algebras the situation is more complicated. In the case of solvable Lie algebras we can not affirm innerness or outerness of derivations definitely.

In the paper \cite{xosiyat} study of solvable Lie algebras for which the dimension of the complementary space is equal to the number of generators of the nilradical.
It is easy to see that all derivetions of the algebra are inner.
In the present note we prove that it is a unique case when solvable Lie algebras with nilradical of maximal rank have no outer derivations. In other words, we prove that any solvable Lie algebra with nilradical of maximal rank and complementary subspace to nilradical of dimension less than rank admits an outer derivation.

In order to begin our study we recall the description of solvable Lie algebras whose nilradical has maximal rank (that is, maximal torus of the nilradical has dimension equal to the number of generator basis elements of nilradical) and dimension of complementary subspace to nilradical is equal to the rank of nilradical.

\begin{thm}\label{thm1}\cite{xosiyat}
Let $R=N\oplus Q$ be a solvable Lie algebra such
that $dimQ=dim N/N^2=k.$ Then $R$ admits a basis $\{e_1, e_2, \dots, e_k, e_{k+1}, \dots,
e_n, x_1, x_2, \dots, x_k\}$ such that the table of multiplications of $R$ in the basis has the following form:
$$\left\{\begin{array}{ll}
[e_i,e_j]=\sum\limits_{t=k+1}^{n}\gamma_{i,j}^te_t,& 1\leq i, j\leq n,\\[1mm]
[e_i,x_i]=e_i,& 1\leq i\leq k,\\[1mm]
[e_i,x_j]=\alpha_{i,j}e_i,& k+1\leq i\leq n,\ \ 1\leq j\leq k,\\[1mm]
\end{array}\right.$$
where $\alpha_{i,j}$ is the number of entries of a generator basis
element $e_j$ involved in forming  non generator basis element
$e_i$.
\end{thm}

\begin{rem} Structure constants $\alpha_{i,j}$ are roots in the decomposition of $N$ with respect to its maximal torus.
\end{rem}

Further we shall also use the following result.

\begin{prop}\label{prop1} \cite{leger} Let $L$ be a Lie algebra over a field of characteristic $0$ such
that $Der(L)=Inder(L).$ If the center of the Lie algebra $L$ is non trivial, then $L$ is not solvable
and the radical of $L$ is nilpotent.
\end{prop}

%





Let $R$ be a solvable Lie algebra with nilradical $N$ and complementary subspace to nilradical $Q$, then $R=N\oplus Q$ as a direct sum of vector spaces.

We set
$$N_{max}=\{ N  \  | \ \ \mbox{there exists solvable } \ R \  \mbox{such that } dimQ=dimN/N^2\}.$$

\begin{thm}
Let $R=N_{max}\oplus Q$ is solvable Lie algebras such that $dimQ < dimN/N^2$. Then $R$ admits an outer derivation. \end{thm}
\begin{proof} Set $dimQ=s$ and $dimN/ N^2=k$ and consider two possible cases.

\emph{\textbf{Case 1.}} Let  $Center(R)=\{0\}$. Since $ad(R)$ is a solvable Lie algebra,
$ad_x$ for $x \in Q$ has upper triangle form and we can write as follows:
$$ad_{x_i}=d_i+d_{n_i},\ \quad 1\leq i\leq s,$$
where $d_i: R\rightarrow R$ is a diagonal derivation and $d_{n_i}:R\rightarrow R$ is nilpotent derivation whose matrix realization has is strictly upper triangle form.

If there exists $i_0$ such that $d_{n_{i_0}}\notin InDer(N)$. Then $d_{i_0}$ is outer derivation.

Suppose now that $d_{n_{i}}\in InDer(N)$ for any $1\leq i\leq s,$ that is, there are $z_i \in N$ such that $ad_{z_i}=d_{n_{i}}$.
Consequently, $ad_{x_i}-ad_{z_i}=ad_{x_i-z_i}$ lie in a maximal torus of $N$ (denoted by $Tor_{max}$). Since $dim Tor_{max}=k>s$, then there exists $d'\in Tor_{max}\setminus Span\{ad_{x_1-z_1},ad_{x_2-z_2},\ldots, ad_{x_s-z_s}\}.$

Taking the change of basis $Q$ as follows $x_i'=x_i-z_i, \ 1\leq i \leq s.$ Then $d'(Q)=0$.

From the equality
$$[z,[x_i',x_j']]=[[z,x_i'],x_j']-[[z,x_j'],x_i']=(ad_{x_j'}\circ ad_{x_i'}-ad_{x_i'}\circ ad_{x_j'})(z)=0$$
for any any $z\in R, \ 1\leq i,j\leq s$ we conclude that $[x_i',x_j']\in Center(R)=\{0\}$, hence $[Q,Q]=0$.

From the equalities
$$d'([n,x])-[n,d'(x)]-[d'(n),x]=d'([n,x])-[d'(n),x]=-d'(ad_{x}(n))+ad_{x}(d'(n))=[ad_{x},d'](n)=0$$
for any $n\in N, x\in Q$ and
$$d'([x_i',x_j'])-[x_i',d'(x_{j}')]-[d'(x_i'),x_j']=0,$$ we derive $d'\in Der(R)$.

Thus, we obtain that $d'$ is outer derivation(otherwise we get a contradiction with condition $dim Q=s$).

\emph{\textbf{Case b.}}  Let $Center(R)\neq \{0\}$. Suppose that $Der(R)=Inner(R)$, then by Proposition \ref{prop1} we conclude that algebra $R$ is nilpotent, that is a contradiction. Therefore, $Inner(R)\subsetneqq Der(R)$.
\end{proof}

\textbf{Acknowledgement.}
The authors are grateful to Professor B.A. Omirov for pointing out the property of considered algebras.

\end{document}